\documentclass[11pt, oneside]{amsart}

\usepackage{amsmath,ifthen, amsfonts, amssymb,
srcltx, 
   amsopn,tikz , url,
   }
\usetikzlibrary{decorations.markings,arrows}

\usepackage{graphicx}

\newcommand{\showcomments}{yes}
\renewcommand{\showcomments}{no}

\newsavebox{\commentbox}
\newenvironment{com}%
{\ifthenelse{\equal{\showcomments}{yes}}%
{\footnotemark
        \begin{lrbox}{\commentbox}
        \begin{minipage}[t]{1.25in}\raggedright\sffamily\tiny
        \footnotemark[\arabic{footnote}]}
{\begin{lrbox}{\commentbox}}}%
{\ifthenelse{\equal{\showcomments}{yes}}%
{\end{minipage}\end{lrbox}\marginpar{\usebox{\commentbox}}}
{\end{lrbox}}}

\newtheorem{thm}{Theorem}[section]
\newtheorem{lem}[thm]{Lemma}

\newtheorem{cor}[thm]{Corollary}

\theoremstyle{definition}
\newtheorem{defn}[thm]{Definition}
\newtheorem{rem}[thm]{Remark}
\newtheorem{exmp}[thm]{Example}

\newtheorem{construction}[thm]{Construction}

\DeclareMathOperator{\Aut}{Aut}

\DeclareMathOperator{\diam}{diam}

\DeclareMathOperator{\stab}{Stabilizer}
\DeclareMathOperator{\Min}{Min}

\newcommand{\field}[1]{\mathbb{#1}}
\newcommand{\integers}{\ensuremath{\field{Z}}}

\newcommand{\dist}{\textup{\textsf{d}}}

\newcommand{\systole}[1]{\ensuremath{\| #1 \|}}
\newcommand{\nclose}[1]{\ensuremath{\langle\!\langle#1\rangle\!\rangle}}

\begin{document}

\title{Cubulating Small Cancellation Free Products}

\author{Kasia Jankiewicz}
\author[D.~T.~Wise]{Daniel T. Wise}
	\address{Dept. of Math.\\
			University of California\\
			Santa Cruz, USA 95064}
	\email{kasia@ucsc.edu}
	\address{Dept. of Math. \& Stats.\\
                    McGill University \\
                    Montreal, Quebec, Canada H3A 0B9}
           \email{wise@math.mcgill.ca}

\date{\today}
\thanks{The first author was supported by the NSF grant DMS-2105548/2203307}
\thanks{The second author was supported by NSERC}

\begin{com}
{\bf \normalsize COMMENTS\\}
ARE\\
SHOWING!\\
\end{com}

\begin{abstract}
We give a simplified approach to the cubulation of small-cancellation quotients of free products of cubulated groups.
We construct fundamental groups of compact nonpositively curved cube complexes that do not virtually split.
 \end{abstract}

\maketitle

\section{Introduction}

Martin and Steenbock recently showed that a small-cancellation quotient of a free product of cubulated groups is cubulated
\cite{MartinSteenbock2016}. In this paper we revisit their theorem in a slightly weaker form, and reprove it in a manner that capitalizes on the available technology. Combined with an idea of Pride's about small-cancellation groups that do not split, we answer a question posed to us by Indira Chatterji by constructing an example
of a compact nonpositively curved cube complex $X$ such that $\pi_1X$ is nontrivial but does not virtually split.

Section~\ref{sec:background} recalls the definitions and theorems that we will use from cubical small-cancellation theory. Section~\ref{sec:relative cocompactness} recalls properties of the dual
cube complex in the relatively hyperbolic setting. Section~\ref{sec:small can free prod} recalls the definition of small-cancellation over free products, and describe associated cubical presentations. Section~\ref{sec:pride} reproves Pride's result about small-cancellation groups that do not split. Section~\ref{sec:main},
relates small-cancellation over free products to cubical small-cancellation theory,
and proves our main result which is Theorem~\ref{thm:main}.
Finally, Section~\ref{sec:doesn't split} combines Pride's method with Theorem~\ref{thm:main} to provide
cubulated groups that do not virtually split in Example~\ref{exmp:nonsplitting cubical group}.

\section{Background on Cubical Small Cancellation}\label{sec:background}
\subsection{Nonpositively curved cube complexes}
We shall assume that the reader is familiar with \emph{CAT(0) cube complexes} which are CAT(0) spaces
having cell structures, where each cell is isometric to a cube. We refer the reader to \cite{BridsonHaefliger, Sageev95, Leary_KanThurston, WiseIsraelHierarchy}. A \emph{nonpositively curved cube complex} is a cell-complex $X$ whose universal cover $\widetilde X$ is a CAT(0) cube complex. A \emph{hyperplane}  $\widetilde U$ in $\widetilde X$ is a subspace
whose intersection with each $n$-cube $[0,1]^n$ is either empty or consists of the subspace where exactly one coordinate is restricted to $\frac12$.
For a hyperplane $\widetilde{U}$ of $\widetilde{X}$, we let $N(\widetilde{U})$ denote its \emph{carrier}, which is the union of all closed cubes intersecting $\widetilde{U}$. The hyperplanes $\widetilde U$ and $\widetilde V$ \emph{osculate} if  $N(\widetilde U)\cap N(\widetilde V)\neq \emptyset$ but
$\widetilde U \cap \widetilde V=\emptyset$.
We will use the \emph{combinatorial metric} on a nonpositively curved cube complex $X$,
so the distance between two points is the length of the shortest combinatorial path connecting them.
The \emph{systole} $\systole{X}$ is the infimal length of an essential combinatorial closed path in $X$.
A map $\phi:Y\rightarrow X$ between nonpositively curved cube complexes is a \emph{local isometry}
if $\phi$ is locally injective, $\phi$ maps open cubes homeomorphically to open cubes,
and whenever $a,b$ are concatenatable edges of $Y$,
 if $\phi(a)\phi(b)$ is a subpath of the attaching map of a 2-cube of $X$,
 then $ab$ is a subpath of a 2-cube in $Y$.

\subsection{Cubical presentations and Pieces}
\begin{defn}
A \emph{cubical presentation} $\langle X \mid Y_1, \ldots, Y_m \rangle$ consists of a nonpositively curved cube complex $X$, and a set of local isometries $Y_i \looparrowright X$ of nonpositively curved cube complexes. We use the notation $X^*$ for the cubical presentation above. As a topological space, $X^*$ consists of $X$ with a cone on $Y_i$ attached to $X$ for each $i$.
\end{defn}
We often consider the universal cover $\widetilde{X^*}$, whose \emph{cubical part} is the preimage of $X$ under the covering map.
The cubical part serves as a ``Cayley graph", whose ``relators" are the cones.

\begin{defn}
A \emph{cone-piece} of $X^*$ in $Y_i$ is a component of $\widetilde{Y}_i \cap \widetilde{Y}_j$, where
$\widetilde{Y}_i$ is a lift of $Y_i$ to the universal cover $\widetilde {X^*}$, excluding the case where $i = j$.
A \emph{wall-piece} of $X^*$ in $Y_i$ is a component of $\widetilde{Y}_i \cap N(\widetilde{U})$, where $\widetilde{U}$ is a hyperplane that is disjoint from  $\widetilde{Y}_i$.
For a constant $\alpha > 0$, we say $X^*$ satisfies the $C'(\alpha)$ \emph{small-cancellation} condition if
$\diam (P) < \alpha \systole{Y_i}$
for every cone-piece or wall-piece $P$ involving $Y_i$.
\end{defn}

When $\alpha$ is small, the quotient $\pi_1X^*$ has good behavior.
For instance, when $X^*$ is $C'(\frac{1}{12})$ then each immersion $Y_i \looparrowright X$ lifts to an embedding
  $Y_i \hookrightarrow \widetilde {X^*}$. This is proven in \cite[Thm~4.1]{WiseIsraelHierarchy},
  and we also refer to \cite{Jankiewicz17} for analogous results at $\alpha=\frac{1}{9}$.

\subsection{The $B(8)$ condition}
We now describe a special case of the B(8) condition within the context of $C'(\alpha)$ 
metric small-cancellation.
A \emph{piece-path} in $Y$ is a path in a piece of $Y$.

\begin{defn} \label{Def:B6}
A cubical presentation $X^*$ satisfies the \emph{$B(8)$ condition} if there is a wallspace structure on each $Y_i$ as follows:

\begin{enumerate}
\item \label{B6:wallspace} The collection of hyperplanes of each $Y_i$ are partitioned into classes
such that no two hyperplanes in the same class cross or osculate, and the union $U = \cup U_k$ of the hyperplanes in a class forms a \emph{wall} in the sense that $Y_i - U$ is the disjoint union of a left and right halfspace.
\item \label{B6:homotopy} If $P$ is a path that is the concatenation of at most
 8 piece-paths
 and $P$ starts and ends on the carrier
$N(U)$ of a wall then $P$ is path-homotopic into $N(U)$.
\item \label{B6:aut} The wallspace structure is preserved by the group $\Aut(Y_i\rightarrow X)$ which consists of automorphisms $\phi:Y_i\to Y_i$ such that $\begin{array}{ccccc}Y_i&&\longrightarrow &&Y_i\\&\searrow&&\swarrow&\\&&X&&
\end{array}$ commutes.
\end{enumerate}
\end{defn}

\subsection{Properness Criterion}
A \emph{closed-geodesic} $w\rightarrow Y$ in a nonpositively curved cube complex, is a combinatorial immersion
of a circle whose universal cover $\tilde w$  lifts to a combinatorial geodesic $\tilde w\rightarrow \widetilde Y$
in the universal cover of $Y$.


We quote the following criterion from \cite[Thm 3.5]{FuterWise2021}.
  \begin{thm}\label{Thm:C'18Proper}
Let $X^* = \langle X \mid Y_1, \ldots, Y_k \rangle$ be a cubical presentation. Suppose $X$ is compact, and each $Y_i$ is compact and deformation retracts to a closed combinatorial geodesic $w_i$. Additionally, suppose that for every hyperplane $U$ of $Y_i$  
the complement $Y_i\setminus U$ is contractible, 
 and $U$ has an embedded carrier with $\diam N(U)<\frac 1 {20}\systole{Y_i} $. If $X^*$ is $C'(\frac{1}{20})$ then $X^*$ is $B(8)$ and  $\pi_1 X^*$ acts properly and cocompactly on the CAT(0) cube complex dual to the wallspace on $\widetilde{X^*}$.

Moreover, if each $\pi_1 Y_i \subset \pi_1X$ is a maximal cyclic subgroup,
 then $\pi_1X^*$ acts freely and cocompactly on the associated dual CAT(0) cube complex.
\end{thm}

The wallspace that is assigned to each $Y_i$ in the above theorem
has a wall for hyperplanes dual to pairs of antipodal edges in $w_i$. 
(The complex $X$ is subdivided to ensure that each $|w_i|$ is even.)
\subsection{The wallspace structure}\label{sec:wallspace structure}
\begin{defn} [The walls]
When $X^*$ satisfies the $B(8)$ condition, $\widetilde{X^*}$ has a \emph{wallspace} structure which we now briefly describe:
Two hyperplanes $H_1,H_2$ of $\widetilde{X^*}$ are \emph{cone-equivalent} if $H_1\cap Y_i$ and $H_2\cap Y_i$ lie in the same wall of $Y_i$ for some lift $Y_i\hookrightarrow \widetilde{X^*}$.
Cone-equivalence generates an equivalence relation on the collection of hyperplanes of $\widetilde{X^*}$.
A \emph{wall} of $\widetilde{X^*}$ is the union of all hyperplanes in an equivalence class.
When $X^*$ is $B(8)$, the hyperplanes in an equivalence class are disjoint, and a wall $w$ can
be regarded as a wall in the sense that $\widetilde{X^*}$ is the union of two halfspaces meeting along $w$.
 \end{defn}
\begin{lem}\label{lem:wall cone intersection}
Let $W$ be a wall of $\widetilde{X^*}$.
Let $Y\subset \widetilde{X^*}$ be a lift of some cone $Y_i$ of $X^*$.
Then either $W\cap Y=\emptyset$ or $W\cap Y$ consists of a single wall of $Y$.
\end{lem}

The \emph{carrier} $N(W)$ of a wall $W$ of $\widetilde{X^*}$
consists of the union of all carriers of hyperplanes of $W$ together with all cones intersected
 by hyperplanes of $W$.
The following appears as \cite[Cor~5.30]{WiseIsraelHierarchy}:

\begin{lem}[Walls quasi-isometrically embed]\label{lem:wallsQIembed}
Let $X^*$ be $B(8)$.
Suppose that pieces have uniformly bounded diameter.
Then for each wall $W$, the map $N(W)\rightarrow \widetilde{X^*}$ is a quasi-isometric embedding with uniform quasi-isometry constants.
\end{lem}

We will need the following result of Hruska which is proven in \cite[Thm~1.5]{HruskaRelQC}:
\begin{thm}\label{thm:qiembedded implies rqc}
Let $G$ be a f.g.\ group that is hyperbolic relative to $\{G_i\}$.
Let $H\subset G$ be a f.g.\ subgroup that is quasi-isometrically embedded.
Then $H\subset G$ is relatively quasiconvex.
\end{thm}

\section{Relative Cocompactness}
\label{sec:relative cocompactness}

The following is a simplified restatement of \cite[Thm~7.12]{HruskaWiseAxioms}
in the case $\heartsuit=\star$. We focus it on our application where the wallspace arises from a cubical presentation.
\newcommand{\neb}{\ensuremath{\mathcal N}}
We use the notation $\neb_d(S)$ for the closed $d$-neighborhood of $S$.
\footnote{There is a small misstatement in \cite[Thm~7.12]{HruskaWiseAxioms}, as it requires
that $r\geq r_0$  for some constant $r_0$.}

\begin{thm}
\label{thm:RelCocompactWallspace}
Consider the wallspace $(\widetilde{X^*},\mathcal W)$.
Suppose $G$ acts properly and cocompactly on the cubical part of $\widetilde{X^*}$ preserving both its metric and wallspace structures,
and the action on $\mathcal W$ has only finitely many $G$--orbits of walls.
Suppose $\stab(W)$ is relatively quasiconvex and acts cocompactly on $W$ for each wall $W\in \mathcal W$.
Suppose $G$ is hyperbolic relative to $\{G_1,\ldots, G_r\}$.
For each $G_i$ let $\widetilde X_i \subset \widetilde{X^*}$ be a
nonempty $G_i$--invariant $G_i$--cocompact subspace.
Let $C(\widetilde{X^*})$ be the cube complex dual
 to $(\widetilde{X^*},\mathcal W)$ and for each $i$ let $C(\widetilde X_i)$
be the cube complex dual to $(\widetilde{X^*},\mathcal W_i)$
where $\mathcal W_i$ consists of all walls $W$ with the property that
$\diam \big(W\cap \neb_d(\widetilde X_i)\big) = \infty$ for some $d=d(W)$.

Then there exists a compact subcomplex $K$ such that
 $C(\widetilde{X^*})= GK \cup \bigcup_i GC(\widetilde X_i)$.
Hence $G$ acts cocompactly on $C(\widetilde{X^*})$ provided that each $C(\widetilde X_i)$ is $G_i$-cocompact.
\end{thm}

In our application of Theorem~\ref{thm:RelCocompactWallspace}, 
$X$ is a ``long'' wedge of cube complexes $X_1, \dots, X_r$ (see Construction~\ref{construction} for the definition) 
and $\widetilde X_i$ is a lift of the universal cover of $X_i$ to $\widetilde{X^*}$. 
The wallspace structure of $X^*$ is described in Section~\ref{sec:wallspace structure} 
(see also Lemma~\ref{lem:still B6 with peripheries}). 
We will be able to apply Theorem~\ref{thm:RelCocompactWallspace} 
because the cube complex $C_\star(\widetilde X_i)$
will be $G_i$-cocompact for the following reason:

\begin{lem}\label{lem:no extra big walls}
Let $G$, $(X^*,\mathcal W)$ be as in Theorem~\ref{thm:RelCocompactWallspace} and suppose that $X$ satisfies $C'(\frac1 {16})$. 
Additionally assume that each $\widetilde X_i$ has the property that if $s$ is a square with an edge in $\widetilde X_i$ then $s\subset \widetilde X_i$.
Let $W$ be a wall of $\widetilde{X^*}$.
Suppose $\diam(W\cap \neb_d(\widetilde X_i))=\infty$ for some $i,d$.
Then $W$ contains a hyperplane of $\widetilde X_i$. 
Hence $C_\star(\widetilde X_i) = \widetilde X_i$ for each $i$.
\end{lem}
\begin{proof}
%


Suppose $\diam(N(W)\cap \neb_d(\widetilde X_i))=\infty$.
By cocompactness of the action $\stab(W)$ on $N(W)$ and $G_i$ on $\widetilde X_i$
there is an infinite order element $g$ stabilizing both $W$ and $\widetilde X_i$.
\begin{com}
Since there are finitely orbits of $0$-cube in each action, there exists a pair of points in the same pair of orbits.
\end{com}

%
Each $\widetilde X_i \subset \widetilde{X^*}$ is convex by \cite[Lem~3.74]{WiseIsraelHierarchy},
and we may therefore choose a geodesic $\tilde \gamma$ in $\widetilde X_i$ that is stabilized by $g$,
and let $\tilde \lambda$ be a path in $N(W)$ that is stabilized by $g$.
We thus obtain an annular diagram $A$ between closed paths $\gamma$ and $\lambda$
which are the quotients of $\tilde \gamma$ and $\tilde \lambda$ by $\langle g\rangle$.
Suppose moreover that $A$ has minimal complexity among all such choices $(A,\gamma,\lambda)$
where $\gamma\rightarrow X_i$ has the property that $\widetilde \gamma$ is a geodesic,
and $\lambda\rightarrow N(W)$ is a closed path.
By
    \cite[Thm~5.61]{WiseIsraelHierarchy},
  $A$ is a square annular diagram, and we may assume it is has no spur.
 Note that \cite[Thm~5.61]{WiseIsraelHierarchy} requires ``tight innerpaths'' which holds at $C'(\frac{1}{16})$ by \cite[Lem~3.70]{WiseIsraelHierarchy}.

Observe that if $s$ is a square with an edge in $\widetilde X_i$,
then $s\subset \widetilde X_i$.
Consequently, the minimality of $A$ ensures that $A$ has no square,
and so $\gamma = A = \lambda$.

There are now two cases to consider:
Either $\tilde \lambda \subset N(U)$ for some hyperplane $U$ of $W$,
or $\tilde \lambda$ has a subpath $u_1y_ju_2$ traveling along $N(U_1), Y_j, N(U_2)$,
where $U_1, U_2$ are distinct hyperplanes of $W$, and $U_1,U_2$ intersect the cone $Y_j$ in antipodal hyperplanes.

In the latter possibility the $B(8)$ condition is contradicted for $Y_j$, since $\widetilde X_i \cap Y_j$
contains  the single piece-path $y_j$ which starts and ends on carriers of distinct hyperplanes of the same wall of $Y_j$.

In the former possibility, $N(U) \cap\widetilde X_i\neq\emptyset$,
and so the above square observation ensures that $N(U)\subset \widetilde X_i$.
Hence $W$  intersects $\widetilde X_i$ as claimed.
\end{proof}

\begin{exmp}
Consider the quotient: $G=\integers^2*\integers^2/\nclose{w_1,w_2}$,
with the following presentation for some number $m>0$:
$$\left\langle \ \langle a,b\mid aba^{-1}b^{-1}\rangle \ * \ \langle c,d \mid cdc^{-1}d^{-1}\rangle  \ \Big| \  a^1c^1a^2c^2\cdots a^{m}c^{m}, \,b^1d^1b^2d^2\cdots b^{m}d^{m}\right\rangle$$
Note that each piece consists of at
most 2 syllables, whereas the syllable length (see Definition~\ref{defn:small can free prod}) of each relator is $2m$.
Hence the $C'_{*}(\frac{1}{m-1})$ small-cancellation condition over free products is satisfied.
See Definition~\ref{defn:small can free prod}.

The associated space $X$ is the long wedge (see Construction~\ref{construction}) of two tori $X_1, X_2$ corresponding to $\langle a,b\rangle$ and $\langle c,d\rangle$.
For $i\in\{1,2\}$, let $Y_i$ be a square complex built out from an alternating sequence of rectangles and arcs as in Figure~\ref{fig:Necklace}.

The cube complex dual to $\widetilde{X^*}$ has $\frac{m(m+1)}{2}$-dimensional cubes
 arising from the cone-cells $Y_1$ and $Y_2$. More interestingly, the cube complex dual to $(\widetilde{X^*}, \mathcal W_1)$
\begin{com}
(We use THE SAME SET but only some of the halfspaces.  That explains how it embeds...)
\end{com}
where $\mathcal W_1$ consists of the walls intersecting a copy of $\widetilde X_1$, has dimension~$2m$.
This is because all hyperplanes dual to the path $a^m$ cross each other because of $Y_1$
 and likewise all hyperplanes dual to the path $b^m$ cross each other because of $Y_2$,
 and every hyperplane dual to the path $a^m$ crosses every hyperplane dual to the path $b^m$
 because $\widetilde X_1$ is a 2-flat.
\end{exmp}
\begin{figure}\centering
\includegraphics[width=.25\textwidth]{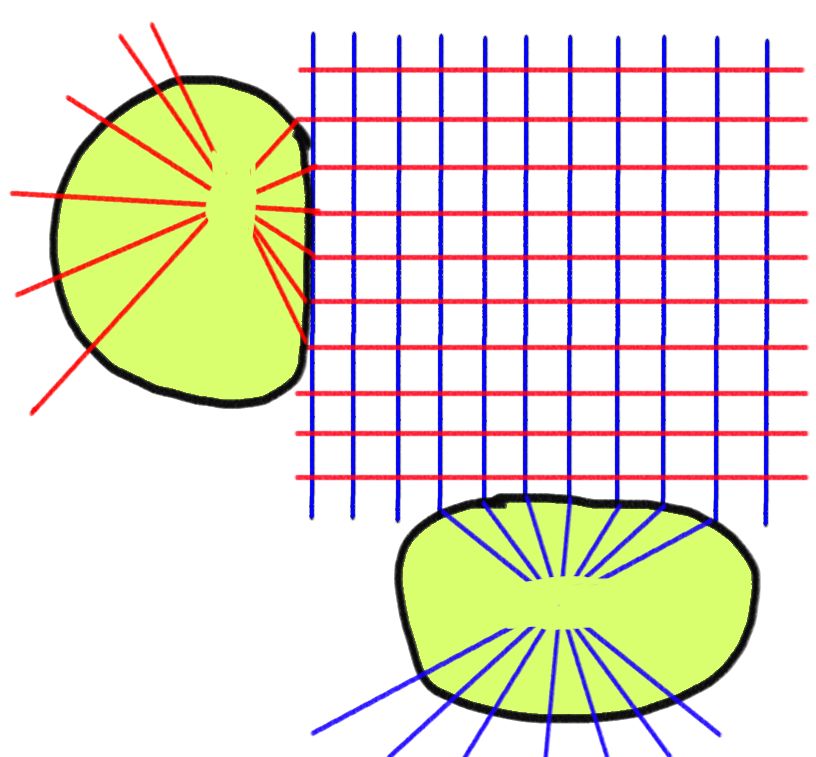}
\caption{ \label{fig:IncreasingDimension}
The walls associated to a 13-cube in the cubulation of a flat.
}
\end{figure}

\section{Small cancellation over free products}\label{sec:small can free prod}
\begin{defn}[The $C_*'(\frac 1 n)$ small cancellation over a free product]\label{defn:small can free prod}
Every element $R$ in the free product $G_1*\cdots *G_r$ has a unique \emph{normal form} which is a word $h_1\cdots h_n$ where each $h_i$ lies in a factor of the free product and $h_i$ and $h_{i+1}$ lie in different factors for $i=1,\dots, n-1$.
The number $n$, which we denote by $|R|_{*}$, is the \emph{syllable length} of $R$.
We say $R$ is \emph{cyclically reduced} if $h_1$ and $h_{n}$ also lie in different factors. We say that $R$ is \emph{weakly cyclically reduced} if $h_n^{-1}\neq h_1$ or if $|R|_{*}\leq 1$.
 We refer to each $h_i$ as a \emph{syllable}.  There is a \emph{cancellation} in the concatenation $P\cdot U$ of two normal forms if the last syllable of $P$ is the inverse of the first syllable of $U$.

Consider a \emph{presentation over a free product} $\langle G_1*\cdots *G_r \mid R_1,\ldots, R_s\rangle$ where each $R_i$ is a cyclically reduced word in the free product.  A word $P$ is a \emph{piece} in $R_i,R_j$ if $R_i, R_j$ have weakly cyclically reduced conjugates $R_i', R_j'$ that can be written as concatenations $P\cdot U_i$ and $P\cdot U_j$ respectively with no cancellations.
 The presentation  is \emph{$C'_{*}(\frac {1}{n})$ small cancellation} if $|P|_{*}<
\frac {1}{n}|R_i'|_{*}$ whenever $P$ is a piece.
\end{defn}

If $G$ is a $C'_{*}(\frac 16)$ small-cancellation quotient of a free product $G_1*\cdots*G_r$ \cite[Cor. 9.4]{LS77}, then each factor $G_i$ embeds in $G$. In particular, $G$ is nontrivial if some $G_i$ is nontrivial.
We quote the following result from \cite{OsinBook06}:
\begin{lem}\label{lem:16relhyp}
Let $G$ be a quotient of $G_1*\cdots*G_r$
arising as a $C'_{*}(\frac{1}{6})$ small-cancellation presentation over a free product.
Then $G$ is hyperbolic relative to $\{G_1,\ldots, G_r\}$.
\end{lem}

\subsection{Cubical presentation associated to a presentation over a free product}

\begin{construction}\label{construction}
Let $T_r$ be the union of directed edges $e_1,\dots, e_r$ identified at their initial vertices.
The \emph{long wedge} of a collection of spaces $X_1,\ldots, X_r$ is obtained from $T_r$
by gluing the basepoint of each $X_j$ to the terminal vertex of $e_j$ . We will later subdivide the edges of $T_r$. Given groups $G_1,\dots G_r$ such that for each $1\leq j \leq r$, let $G_j=\pi_1X_j$ where $X_j$ is a nonpositively curved cube complex, the long wedge $X$ of the collection $X_1,\dots, X_r$ is a cube complex with $\pi_1 X = G_1
*\dots*G_r$.

Given an element $R\in G_1*\dots*G_r$ with $|R|_{*}>1$, there exists a local isometry $Y\to X$ where $Y$ is a compact nonpositively curved cube complex with $\pi_1Y=\langle R\rangle$. Indeed, let $R = h_1h_2\cdots h_t$ where each $h_k$ is an element of some $G_{m(k)}$. For each $k$ let $V_k$ be the compact cube complex that is the combinatorial convex hull of the basepoint $p$ and its translate $h_kp$ in the universal cover $\widetilde X_{m(k)}$. We call $p$ the \emph{initial vertex} of $V_k$ and $h_kp$ the \emph{terminal vertex} of $V_k$. For each $1\leq k\leq t$ let $
\sigma_k$ be a copy of $e_{m(k)}^{-1}e_{m(k+1)}$ where $m(t+1) = m(1)$. Finally we form $Y$ from $\bigsqcup_{k=1}^t V_k$  and $\bigsqcup_{k=1}^{t} \sigma_k$ by gluing the terminal vertex of $V_{k}$ to the initial vertex of $\sigma_k$ and the terminal vertex of $\sigma_k$ to the initial vertex of $V_{k+1}$.
Note that there is an induced map $Y\to X$ which is a local isometry. See Figure~\ref{fig:Necklace}.

\begin{figure}\centering
\includegraphics[width=.7\textwidth]{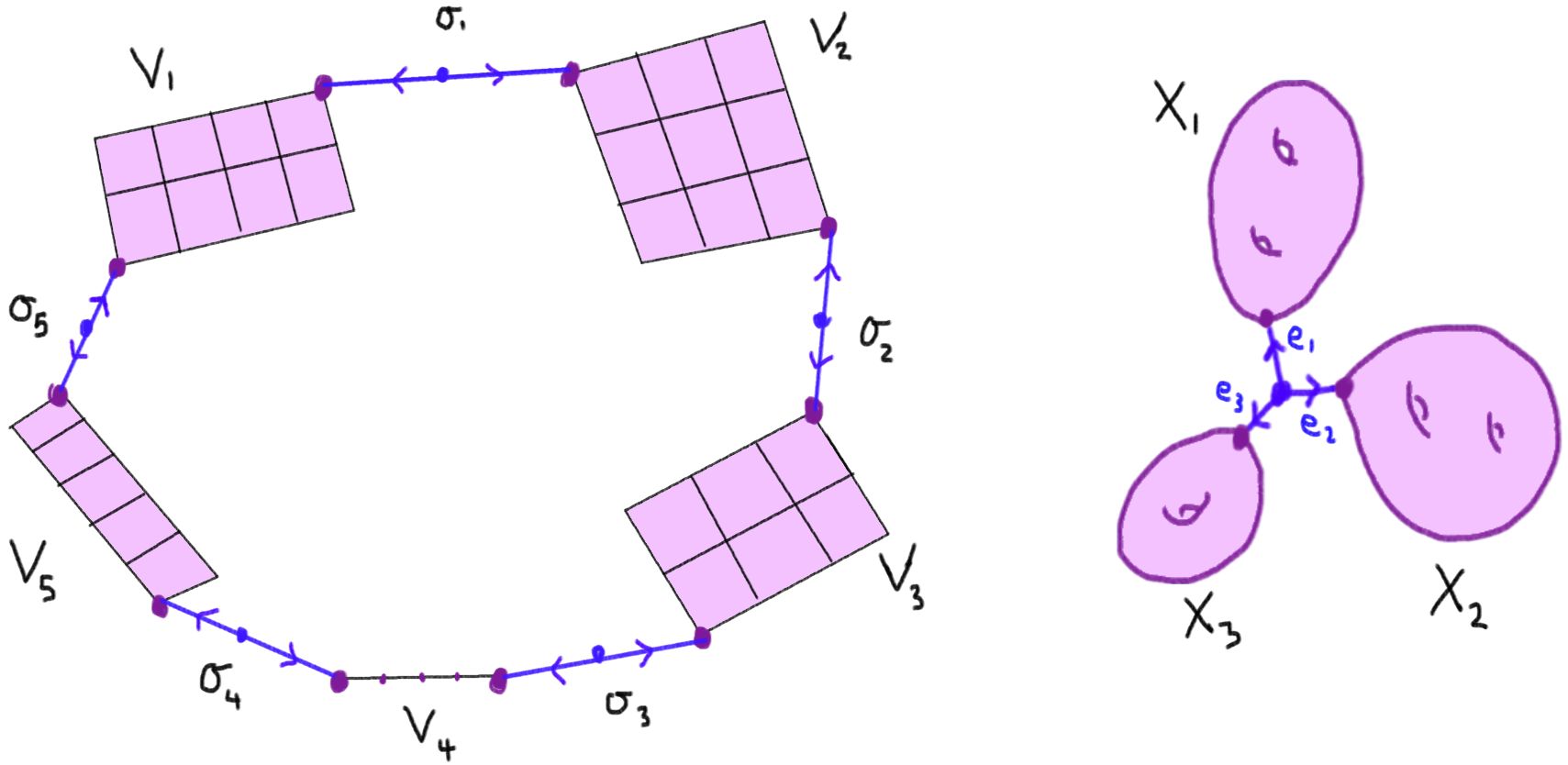}
\caption{ \label{fig:Necklace}
On the right is a long wedge of surfaces. 
On the left is a complex $Y$ mapping to $X$ by a local isometry. 
A relator of syllable length $n$ is represented by such a local isometry having $n$ rectangles.
}
\end{figure}

Given a presentation $\langle G_1,\ldots, G_r \mid R_1,\ldots, R_s\rangle$ over a free product there is an associated cubical presentation $X^*=\langle X\mid Y_1,\dots, Y_s\rangle$ where each $Y_i\to X$ is a local isometry associated to $R_i$ as above.
Finally, any subdivision of the edges $e_1,\dots, e_r$ induces a subdivision of $X$, and accordingly a subdivision of each $Y_i$. We thus obtain a new cubical presentation that we continue to denote by $X^*$.
\end{construction}

\begin{lem}\label{lem:still B6 with peripheries}
Suppose $\langle X\mid Y_1,\ldots, Y_s\rangle$ is $B(8)$ (after subdividing).
And let $\widetilde X_k$ be the universal cover of $X_k$ with the wallspace structure such that each hyperplane is a wall. 
Then $\langle X \mid Y_1,\ldots, Y_s, \widetilde X_1,\ldots, \widetilde X_r\rangle$, where the maps $\widetilde X_j\to X$ are the local isometries factoring as $\widetilde X_j\to X_j\to X$, is $B(8)$.
Moreover, the two wallspace structures can be chosen so that the walls of $\widetilde{X^*}$ induced by the two structures are identical.
\end{lem}
\begin{proof}
We choose the same wallspace structure on each $ Y_i$ as before, and the natural wallspace structure given by the hyperplanes on each $\widetilde X_j$.
The cone-pieces between $\widetilde X_j$ and $Y_i$ are copies of the $V_k$ associated to $X_j$ that appear in $Y_i$,
and hence Condition~\ref{Def:B6}\eqref{B6:homotopy} holds for each $Y_i$ as before.
For each $\widetilde X_j$, Conditions~\ref{Def:B6}\eqref{B6:wallspace} and \ref{Def:B6}\eqref{B6:aut}
hold automatically by our choice of wallspace structure, and Condition~\ref{Def:B6}\eqref{B6:homotopy} holds since $\widetilde X_j$ is contractible.
\end{proof}
\begin{cor}\label{cor:single hyperplane}
For each wall $W$ of $\widetilde{X^*}$, the intersection of $W\cap \widetilde X_j$ is either empty or consists of a single hyperplane.
\end{cor}
\begin{proof}
This follows by combining Lemma~\ref{lem:still B6 with peripheries} and Lemma~\ref{lem:wall cone intersection}.
\end{proof}

\section{Construction of Pride}\label{sec:pride}
The following result was proven by Pride in \cite{Pride83}.
We give a slightly more geometric version of his proof, which was originally proven only for a $C(n)$ presentation instead of a $C'(\frac1n)$ presentation, which we can obtain as in Remark~\ref{rem:C'}.

\begin{lem}\label{lem:pride}Let  $G=\langle x,y\mid R_1, R_2, R_3, R_4, R_5, R_6\rangle$ where the relators $R_i$ are specified below for associated positive integers $\alpha_i, \beta_i, \gamma_i, \delta_i, \rho_i, \sigma_i, \tau_i, \theta_i$ for each $1\leq i\leq k$, and $k\geq 1$. Then $G$ does not split as an amalgamated product or HNN extension.
\begin{align*}
&R_1(x,y) = xy^{\alpha_1}xy^{\alpha_2}\cdots xy^{\alpha_k}\\
&R_2(x,y) = yx^{\beta_1}yx^{\beta_2}\cdots yx^{\beta_k}\\
&R_3(x,y) = x^{\gamma_1}y^{-\delta_1}x^{\gamma_2}y^{-\delta_2}\cdots x^{\gamma_k}y^{-\delta_k}\\
&R_4(x,y) = xy^{\rho_1}xy^{-\rho_1}xy^{\rho_2}xy^{-\rho_2}\cdots xy^{\rho_k}xy^{-\rho_k}\\
&R_5(x,y) = yx^{\sigma_1}yx^{-\sigma_1}yx^{\sigma_2}yx^{-\sigma_2}\cdots yx^{\sigma_k}yx^{-\sigma_k}\\
&R_6(x,y) = (xy)^{\tau_1}(x^{-1}y^{-1})^{\theta_1}(xy)^{\tau_2}(x^{-1}y^{-1})^{\theta_2}\cdots(xy)^{\tau_k}(x^{-1}y^{-1})^{\theta_k}
\end{align*}
\end{lem}

\begin{proof}

Suppose $G = A*_CB$ or $G=A*_C$ and let $T$ be the associated Bass-Serre tree.
Without loss of generality, assume that the translation length of $y$ is at least as large as the translation length of $x$. Choose a vertex $v\in \Min(x)$ for which $\dist_T(y{\cdot} v,v)$ is minimal.

For use in the argument below, given a decomposition
of $w\in G$ as a product $w=w_1w_2\cdots w_\ell$
, the path $[v,w_1{\cdot}v][w_1{\cdot}v, w_1w_2{\cdot}v]\cdots[w_1w_2\cdots w_{\ell-1}{\cdot}v, w{\cdot}v]$ is said to \emph{read} $w$.

We now show that $v\in \Min(y)$. First suppose that $x$, and hence $y$, is a hyperbolic isometry. 
If $v\notin \Min(y)$, i.e.\ $\Min(x)\cap \Min(y) = \emptyset$, then the axes of $x$ and $y$ in $T$ are disjoint, and $v$ is a vertex in the axis of $x$ minimizing the distance between the two axes.
In particular, the concatenation of two nontrivial geodesics $[x^{-1}{\cdot}v,v][v, y{\cdot} v]$ would be a geodesic. See Figure~\ref{fig:minsets intersect}.
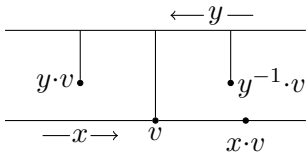
\begin{figure}
\begin{tikzpicture}
\tikzstyle{every node}=[circle, draw, fill,
                        inner sep=0pt, minimum width = 2pt]
\node[label=below:$v$] (v) at (0,0) {};
\node[label=below:$x{\cdot}v$] (xv) at (1.2,0) {};
\node[label=left:$y{\cdot}v$] (yv) at (-1,0.5) {};
\node[label=right:$y^{-1}{\cdot}v$] (y^-1v) at (1,0.5) {};
\draw (2,0) to (v) to (-2,0);
\draw (v) to (0,1.2);
\draw(2,1.2) to (-2,1.2);
\draw (-1,1.2) to (yv);
\draw (1,1.2) to (y^-1v);
\draw[->] (1.3,1.4) to (0.8,1.4) node[draw=none, fill=white] {$y$} to (0.2,1.4);
\draw[->] (-1.5,-0.2) to (-1,-0.2) node [draw=none, fill=white] {$x$} to (-0.5,-0.2);
\end{tikzpicture}
\caption{The case where $\Min(x)\cap\Min(y)=\emptyset$.}\label{fig:minsets intersect}
\end{figure}
 Similarly $[x{\cdot}v,v] [v,y{\cdot} v]$, $[x^{-1}{\cdot}v,v][v,y^{-1}{\cdot} v]$ and $[x{\cdot}v,v][v,y^{-1}{\cdot} v]$ would be geodesics. 
 Consequently, regarding $R_6$ as a product of elements $\{x^{\pm1},y^{\pm1}\}$, we see that the path reading $R_6$ would be a geodesic, which contradicts that $R_6=_G1$.
 Now, suppose that $x$ is elliptic and so $x{\cdot}v=v$. 
 Let $e$ denote the initial edge of $[v, y{\cdot}v]$ and note that $e$ is also the initial edge of $[v, y^{-1}{\cdot}v]$ since $v\notin\Min(y)$. 
 The choice of $v$ implies $x{\cdot}e\neq e$, as otherwise the other endpoint $v'$ of $e$ would satisfy $\dist_T(y{\cdot}v', v')<\dist_T(y{\cdot}v,v)$.
 Thus the concatenation of the nontrivial geodesics $[y^{-1}{\cdot}v,v][v,xy{\cdot}v]$ is a geodesic, and similarly for $[y^{-1}{\cdot}v,v][v, x^{-1}y^{-1}v]$, $[y{\cdot}v,v][v, xy{\cdot}v]$ and $[y{\cdot}v,v][v, x^{-1}y^{-1}v]$. 
It follows that regarding $R_6$ as a product of elements $\{xy,x^{-1}y^{-1}\}$, the path reading $R_6$ is a geodesic, which contradicts that $R_6=_G1$. Therefore $v\in \Min (y)$.

Since $v\in \Min(x)\cap\Min(y)$, the element $y$ is a hyperbolic isometry, because otherwise $x,y$ are elliptic and so $v$ is a global fixed point.
Suppose $x$ is also a hyperbolic isometry.
At least one of $[y^{-1}{\cdot}v,v][v,x{\cdot}v]$ or $[x^{-1}{\cdot} v, v][v,y{\cdot} v]$ is not a geodesic, because otherwise the path reading $R_1$ regarded as a product of $\{x^{\pm1},y^{\pm1}\}$ would be a geodesic. Consequently, both $[x{\cdot} v,v][v, y{\cdot} v]$ and $[x^{-1}{\cdot} v,v][v, y^{-1}{\cdot} v]$ are geodesics, and hence
regarding $R_3$ as a product of elements $\{x^{\pm1},y^{\pm1}\}$,
 the path reading $R_3$ must be a geodesic, which is a contradiction. Thus, $x$ is an elliptic isometry.

Let $e_+$ and $e_-$ denote the initial edges of  $[v, y{\cdot}v]$ and $[v,y^{-1}{\cdot}v]$ respectively. See Figure~\ref{fig:initial edges}.
Let us explain why $x{\cdot}e_+ = e_-$. 
Otherwise $[y^{-1}{\cdot}v,v][v,xy{\cdot}v]$ would be a geodesic since the last edge of $[y^{-1}{\cdot}v,v]$ is $e_-$ and the first edge of $[v,xy{\cdot}v]$ is $x{\cdot}e_+$. 
Likewise, for $n,m>0$ the path $[y^{-n}{\cdot}v,v][v,xy^m{\cdot}v]$ would be a geodesic, and so too would be its translate $[v,xy^n{\cdot}v][xy^n{\cdot}v, xy^nxy^m{\cdot}v]$. 
Regarding $R_1$ as a product  $(xy^{\alpha_1})(xy^{\alpha_2})\cdots (xy^{\alpha_k})$,
the path reading $R_1$ would be a geodesic, contradicting $R_1=_G1$.

Since $x{\cdot}e_+ = e_-$, neither $e_-$ nor $e_+$ is fixed by $x$. 
For any $n,m>0$ the last edge of $[y^n{\cdot}v,v]$ is $e_+$ and the first edge of $[v,xy^m{\cdot}v]$ is $x{\cdot}e_+=e_-\neq e_+$, and so the path $[y^n{\cdot}v,v][v,xy^m{\cdot}v]$ is a geodesic, and so is $[v,y^{-n}{\cdot}v][y^{-n}{\cdot}v,y^{-n}xy^m{\cdot}v]$. 
Similarly, the last edge of $[y^{-n}{\cdot}v,v]$ is $e_-$ and the first edge of $[v,xy^{-m}{\cdot}v]$ is $x{\cdot}e_-\neq e_-$, and so the path $[y^{-n}{\cdot}v,v][v,xy^{-m}{\cdot}v]$ is a geodesic as is $[v,xy^{n}{\cdot}v][xy^{n}{\cdot}v,xy^{n}{\cdot}xy^{-m}{\cdot}v]$.
Regarding $R_4$ as a product $(xy^{\rho_1})(xy^{-\rho_1}) \cdots  (xy^{\rho_k})(xy^{-\rho_k})$, we see that the path reading $R_4$ is a geodesic, contradicting $R_4=_G1$. 
This completes the proof.
\begin{com} (This proof did not use that the abelianization was finite to exclude the possibility of an HNN extension. It does not use that aspect of the form of $R_4$, but rather that $R_4$ has alternating positive and negative powers of $y$.)\end{com}
\end{proof}

\begin{rem}\label{rem:C'}
In the context of Lemma~\ref{lem:pride},  for each $n$ there are choices of
$k$ and $\{\alpha_i, \beta_i, \gamma_i, \delta_i, \rho_i, \sigma_i, \tau_i, \theta_i \ : \ 1\leq i\leq k\}$,
such that the presentation is $C'(\frac{1}{n})$.

Given $n>1$, let $k=3 n$ and choose $8k$ numbers $\alpha_i, \beta_i, \gamma_i, \delta_i, \rho_i, \sigma_i, \tau_i, \theta_i$ that are all different and lie between $50 n$ and $75 n$. Then any piece $P$ in $R_i$ where $i\neq 6$ is of the form $x^{l}yx^{m}$ or $y^{l}xy^{m}$ for some $l,m$ (possibly $0$). Thus $|P|\leq l+m+1\leq 150 n+1$. We also have $|R_i|\geq (k+1)50 n = (3n+1)50 n$ and so
$|P|\leq \frac 1n (150n+1)n\leq \frac 1 n |R_i|$. If $P$ is a piece in $R_6$, then $P$ is of the from $(xy)^l(x^{-1}y^{-1})^m$ and so $|P|\leq 2(l+m)\leq 300n$. We also have $|R_6|= 2(\tau_1+\theta_1+\tau_2+\dots+\theta_k) \geq 2(2k)50n = 600 n^2$. Hence $|P|\leq \frac 1 n |R_6|$.
\end{rem}

\begin{figure}
\begin{tikzpicture}
\tikzstyle{every node}=[circle, draw, fill,
                        inner sep=0pt, minimum width = 2pt]
\node[label=below:$v_{}$] (v) at (0,0) {};
\node (-) at (-1,0) {};
\node (+) at (1,0) {};
\node (x+) at (0.9, 0.4) {};
\node[label=$y^{-1}{\cdot}v$] (y^-1v) at (-3,0) {};
\node[label=$y{\cdot}v$] (yv) at (3,0) {};
\node[label=$xy{\cdot}v$] (xyv) at (2.7, 1.2) {};
\tikzstyle{every node}=[]
\draw (v) -- node[above] {$x{\cdot}e_+$} (x+) -- (xyv);
\draw (v) -- node[below] {$e_-$} (-) -- (-2,0) -- (y^-1v);
\draw (v) -- node[below] {$e_+$} (+) -- (2,0) -- (yv);
\end{tikzpicture}\caption{If $x{\cdot}e_+ \neq e_-$ then $[y^{-1}{\cdot}v,v][v,xy{\cdot}v]$ is a geodesic.}\label{fig:initial edges}
\end{figure}
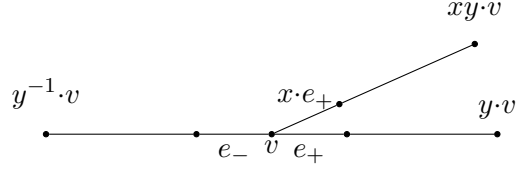

\begin{cor}\label{cor:pride}
Let $G_1,\ldots, G_r$ be nontrivial groups generated by finite sets of infinite order elements,
and suppose $r>1$.
For each $n>0$ there is a finitely related $C_{*}'(\frac1n)$ quotient
$G$ of $G_1*\cdots*G_r$
that  does not split.
\end{cor}
\begin{proof} Let $S_p$ be the given generating set of $G_p$ for each $p$, and assume no proper subset of $S_p$ generates $G_p$. The desired quotient $G$ arises from a presentation
$\langle G_1 *\cdots * G_r \mid \mathcal R\rangle$, where following Lemma~\ref{lem:pride}, the set of relators is:
$$ \mathcal R \ = \ \left\{\ R_\ell(x,y) \ : \   1\leq \ell \leq 6, \ (x,y) \in S_p\times S_q, 
\text{ where } \ 1\leq p<q\leq r \right\}$$
where $k(x,y)=3n$ for each $(x,y)\in S_p\times S_q,$ and
where the constants $\alpha_i(x,y)$, $\beta_i(x,y)$, $\gamma_i(x,y)$, $\delta_i(x,y)$, $\rho_i(x,y)$, $\sigma_i(x,y)$, $\tau_i(x,y)$, $\theta_i(x,y)$ will be described below.
For each $(x,y)\in S_p\times S_q,$ let $\alpha_i(x,y)$, $\delta_i(x,y)$ and $\rho_i(x,y)$ be distinct integers $>1$ such that $y^{m}\notin \langle z\rangle$ for $m\in \{\alpha_i(x,y),\delta_i(x,y),\rho_i(x,y)\}$ and $z\in S_q-\{y\}$. This is possible because $y$ has infinite order and $y\notin \langle z\rangle$. Similarly, let $\beta_i(x,y)$, $\gamma_i(x,y)$ and $\sigma_i(x,y)$ be distinct integers $>1$ such that $x^m\notin \langle z \rangle$ for $m\in\{\beta_i(x,y),\gamma_i(x,y),\sigma_i(x,y)\}$ and $z\in S_p-\{x\}$. Finally, let $\tau_i(x,y)$ and $ \theta_i(x,y)$ be distinct integers between $10n$ and $20n$.

Having chosen the above constants for each $(x,y) \in S_p\times S_q,$ we now show that the presentation for $G$ is $C'_*(\frac{1}{n})$.
We begin by observing that each $|R_\ell(x,y)|_*\geq 6n$.
Let $P$ be a piece in $R^1=R_{\ell_1}(x_1, y_1)$ and $R^2=R_{\ell_2}(x_2, y_2)$ where $x_1\in S_{p_1}$,
 $y_1\in S_{q_1}$, $x_2\in S_{p_2}$, and $y_2\in S_{q_2}$.
If $\{p_1, q_1\}\neq\{p_2, q_2\}$ then $|P|_*\leq 1$.
 Assume that $\{p_1, q_1\}=\{p_2, q_2\}$. First suppose that $\ell_1\neq 6$, then $|P|_*\leq 3$. Indeed, if $|P|_*\geq 4$ then two consecutive syllables would appear in distinct cyclically reduced forms of relators, which contradicts our choice of the constants. If $\ell_1=6$, then $|P|_*\leq \max\{\tau_i(x,y)\}+\max\{\theta_i(x,y)\}\leq 80 n$. We also have $|R_6(x,y)|_* = 2\left(\tau_1(x,y)+\theta_1(x,y)+\dots+\tau_k(x,y)+\theta_k(x,y)\right)\geq 2(2k)10n = 120n^2$, so $|P|_*\leq \frac 1 n |R_6(x,y)|_*$.

We now show that $G$ does not split as an amalgamated product.
For each $x\in S_p, y\in S_q$ with $p\leq q$ we let $H(x,y) = \langle x,y \mid R_\ell(x,y) : 1\leq \ell \leq 6\rangle$.
By Lemma~\ref{lem:pride}, we see that $H(x,y)$
does not split.
As there is a homomorphism $H(x,y)\rightarrow G$, we deduce that
for any splitting of $G$ as an amalgamated free product $G=A*_C B$,
 the elements $x,y$ are either both in $A$ or both in $B$. 
 Otherwise, the action of $H(x,y)$ on the Bass-Serre tree of $G = A*_CB$ induces a non-trivial splitting. 
 Considering all such pairs $(x,y)$,
  we find that the generators of $G$ are either all in $A$ or all in $B$. Moreover $G$ cannot split as an HNN extension, since the relators $R_4(x,y)$ and $R_5(x,y)$
  show that all generators have finite order in the abelianization of $G$.
\end{proof}
\section{Main theorem}\label{sec:main}
The small cancellation over a free product condition $C'_{*}(\frac{1}{n})$ was defined in Definition~\ref{defn:small can free prod}. We start with the following Lemma.

\begin{lem}\label{lem:free product smallcan to cubical smallcan}
If $\langle G_1,\ldots, G_r \mid R_1,\ldots, R_s\rangle$ is $C'_{*}(\frac{1}{n})$ then for a sufficient subdivision of $e_1,\dots, e_r$ the cubical presentation $X^*$ is $C'(\frac{1}{n})$.
\end{lem}
\begin{proof} Let $X'$ be a subdivision of $X$ induced by a $q$-fold subdivision of each $e_j$. We accordingly let $Y_i'$ be the induced subdivision of $Y_i$,
so $Y_i' =\bigsqcup V_k\cup\bigsqcup \sigma_k$ as in Construction~\ref{construction} and with each $\sigma$-edge subdivided.
 We thus obtain a new cubical presentation $\langle X'\mid Y_1',\dots,Y_s'\rangle$. 
 Since $Y_i$ has $|R_i|_{*}$ $\sigma$-edges, the systole $\systole{Y_i'} = \systole{Y_i} + 2|R_i|_{*}(q-1)$. Note that $\systole{Y_i'}> \sum_{i=1}^{|R_i|_{*}} |\sigma_i| = 2q|R_i|_{*}$
 and so $\systole{Y_i'}> 2(1+\epsilon)q|R_i|_{*}$ for sufficiently small $\epsilon>0$. Let $M_i = \max_k\left\{\diam(V_k)\right\}$. For a wall-piece $P$ we have $\diam(P)<M_i$. Consider a maximal cone-piece $P$ in $Y_i'$, and suppose it  intersects $\ell$ different $V_k$'s and contains
 $\ell'$ different $e_k$ edges. Note that $2\ell\geq \ell'$ since if $P$ starts or ends with an entire $\sigma_k$ arc, then it intersects an additional $V_k$ (possibly trivially). We have $\diam(P)\leq \ell M_i + q\ell'$. When $\ell'>0$, for any $\epsilon>0$ we can choose $q\gg 0$  so that
 $\diam(P)< (1+ \epsilon)q\ell'$.
 Since $P$ corresponds to a length~$\ell$ syllable piece, the $C'_{*}(\frac1n)$ hypothesis implies that $\ell < \frac{1}{n}|R_i|_{*}$, and so
 $\diam(P)< (1+ \epsilon)q\ell'  < 2(1+\epsilon)q(\frac{1}{n}|R_i|_{*}) <\frac{1}{n}\systole{Y_i'}$.
  When $\ell'=0$, then assuming $q>nM_i$ we have $\diam(P)\leq M_i < 2\frac{q}{n}|R_i|_{*}<\frac1n\systole{Y_i'}$.
 \end{proof}

\begin{thm}\label{thm:main}
Suppose $G= \langle G_1,\ldots, G_r \mid R_1,\ldots, R_s\rangle$ satisfies $C'_{*}(\frac{1}{20})$.
If each $G_i$
is the fundamental group of a $[$compact$]$ nonpositively curved cube complex,
then $G$
acts properly $[$and compactly$]$ on a CAT(0) cube complex.

Moreover, $G$ acts freely if each $\langle R_i\rangle$ is a maximal cyclic subgroup.
\end{thm}

\begin{proof}
Let $X^*$ be the associated cubical presentation.
Lemma~\ref{lem:free product smallcan to cubical smallcan} asserts that $X^*$ is $C'(\frac{1}{20})$ after a sufficient subdivision. For each hyperplane $U$ in $Y_i$ we have $\diam(N(U))<\frac 1 {20} \systole{Y_i}$ if the subdivision is sufficient. Theorem~\ref{Thm:C'18Proper} asserts that $\pi_1X^*$ acts freely (or with finite stabilizers if relators are proper powers) on a CAT(0) cube complex $C$ dual to $\widetilde{X^*}$.

Let $X'^*$ be the cubical presentation $\langle X \mid \{Y_i\}, \{\widetilde X_j\}\rangle$.
By Lemma~\ref{lem:still B6 with peripheries},  $X'^*$ satisfies $B(8)$ with our previously chosen wallspace structure on each $Y_i$ and the hyperplane wallspace structure on each $\widetilde X_j$.
Thus by Lemma~\ref{lem:wall cone intersection} each $\widetilde X_j$ in $\widetilde{X^*} = \widetilde {X'^*}$
intersects the walls of $\widetilde{X^*}$ in hyperplanes of $\widetilde X_j$.

Lemma~\ref{lem:16relhyp} asserts that $\pi_1X^*$ is hyperbolic relative to $\{G_1,\ldots, G_r\}$.

The pieces in $X^*=\langle X \mid \{Y_i\} \rangle $
are uniformly bounded since $\diam(Y_i)$ is uniformly bounded.
Thus $N(W)\rightarrow \widetilde{X^*}$ is quasi-isometrically embedded by Lemma~\ref{lem:wallsQIembed}.
Hence $\stab(N(W))$ is relatively quasiconvex with respect to $\{\pi_1X_j\}$ by Theorem~\ref{thm:qiembedded implies rqc}.

Theorem~\ref{thm:RelCocompactWallspace} asserts that $\pi_1X^*$ acts relatively cocompactly on $C$.
Lemma~\ref{lem:no extra big walls} asserts that each $C_\star(\widetilde X_i)=\widetilde X_i$.
Hence if each $X_i$ is compact, we see that $C$ is compact.
\end{proof}

\section{A cubulated group that does not virtually split}\label{sec:doesn't split}
Examples were given in \cite{WiseIsraelHierarchy} of a compact nonpositively curved cube complex $X$
such that $X$ has no finite cover with an embedded hyperplane. It is conceivable that those groups have no (virtual) splitting, but this was not confirmed there.
\begin{exmp}\label{exmp:nonsplitting cubical group}
There exists a nontrivial group $G$ with the following two properties:
\begin{enumerate}
\item $G=\pi_1X$ where $X$ is a compact nonpositively curved cube complex.
\item $G$ does not have a finite index subgroup that splits as an amalgamated product or HNN extension.
\end{enumerate}

 Let $G_1$ be the fundamental group of $X_1$ which is a compact nonpositively curved cube complex with a nontrivial fundamental group but no nontrivial finite cover.
For instance, such complexes were constructed in \cite{Wise96Thesis} or \cite{BurgerMozes97}.
By Corollary~\ref{cor:pride} there exists a $C'_{*}(\frac{1}{20})$ quotient $G$
of the free product $G_1 * \cdots * G_1$ of $r$ copies of $G_1$, such that $G$ does not split. The group $G$ has no finite index subgroups since $G_1 * \cdots * G_1$ has none.
Since $G_1=\pi_1X_1$, by Theorem~\ref{thm:main}, $G$ is the fundamental group of a compact nonpositively curved cube complex.
\end{exmp}

\bibliographystyle{alpha}
\bibliography{../../kasia}

%
%
\end{document}